\documentclass[twoside]{irmaems}
\usepackage{amssymb} 
\usepackage{amsmath} 
\usepackage{latexsym}
\usepackage{amsxtra, amscd, mathrsfs, graphicx, multirow}
\usepackage{makeidx}
\usepackage[all]{xypic}

\renewcommand\theequation{\arabic{section}.\arabic{equation}}

\theoremstyle{definition} 

 \newtheorem{definition}{Definition}[section]

\theoremstyle{plain}      

 \newtheorem{prop}[definition]{Proposition}
 \newtheorem{thm}[definition]{Theorem}
 \newtheorem{cor}[definition]{Corollary}
 \newtheorem{lem}[definition]{Lemma}
 \newtheorem{defn}[definition]{Definition}
 \newtheorem{rem}[definition]{Remark}
 \newtheorem{exs}[definition]{Examples}
 \newtheorem{ass}[definition]{Assumption}

\newcommand{\Hom}{\operatorname{Hom}}

\def\R{{\mathbb R}}
\def\Z{{\mathbb Z}}
\def\C{{\mathbb C}}
\def\Q{{\mathbb Q}}
\def\M{{\mathcal M}}

\def\Ext{\text{\rm Ext}}

\def\Hom{\text{\rm Hom}}
\def\Aut{\mathop{\rm Aut}}

\def\CH{\text{\rm CH}}
\def\G{\Gamma}

\allowdisplaybreaks

\begin{document}

\title{Harmonic volume and its applications}

\author{Yuuki Tadokoro}

\address{
Natural Science Education,\\
Kisarazu National College of Technology, 2-11-1 Kiyomidai-Higashi,\\
Kisarazu, Chiba 292-0041, Japan\\
email:\,\tt{tado@nebula.n.kisarazu.ac.jp}
}

\maketitle

\begin{abstract}
The period is a classical complex analytic invariant for a compact Riemann surface defined by integration of differential 1-forms.
It has a strong relationship with the complex structure of the surface.
In this chapter, we review another complex analytic invariant called the harmonic volume.
It is a natural extension of the period defined using Chen's iterated integrals and captures more detailed information of the complex structure.
It is also one of a few explicitly computable examples of complex analytic invariants.
As an application, we give an algorithm in proving nontriviality for a class of homologically trivial algebraic cycles obtained from special compact Riemann surfaces.
The moduli space of compact Riemann surfaces is the space of all biholomorphism classes of compact Riemann surfaces.
The harmonic volume can be regarded as an analytic section of
a local system on the moduli space.
It enables a quantitative study of the local structure of the moduli space.
We explain basic concepts related to the harmonic volume and its applications of the moduli space.
\end{abstract}



\tableofcontents

\section{Introduction}
\label{Introduction}
Let $\Gamma_g$ be the mapping class group of an oriented closed surface $\Sigma_g$ of genus $g\geq 2$.
This is the group of all isotopy classes of orientation-preserving diffeomorphisms of $\Sigma_g$.
The moduli space\index{moduli space} $\M_g=\mathcal{T}_g/{\Gamma_g}$ of compact Riemann surfaces of genus $g$ is the quotient space of the Teichm\"uller space $\mathcal{T}_g$ by the natural action of
the mapping class group $\Gamma_g$.
It is the space of all biholomorphism classes of compact Riemann surfaces of genus $g$.
The global structure of $\M_g$ is one of the most attractive subjects in mathematics.
For its study, topological methods such as cohomology and homotopy theory are often used.
This comes from the well-known fact that the natural properly discontinuous action of the mapping class group on the Teichm\"uller space $\mathcal{T}_g$ which is contractible yields a rational cohomology equivalence between the moduli space $\M_g$ and the classifying space of the mapping class group $\Gamma_g$.

In contrast, for the analytic approach, the universal family\index{universal family} $\pi: {\mathcal C}_g\to \M_g$ of compact Riemann surfaces of genus $g$ plays an important role.
The Abel-Jacobi map from a compact Riemann surface to its Jacobian variety has been studied by many people.
Here we focus on the fact that the Abel-Jacobi map can be regarded as a natural section of a local system on ${\mathcal C}_g$.
The {\it pointed harmonic volume}, the main theme of this chapter, is an extension of this framework.
We review the pointed harmonic volume and its applications.

Let $H$ be the set of real harmonic 1-forms on a compact Riemann surface with integral periods.
The Hodge theorem gives an identification of $H$ with the first integral cohomology group of $\Sigma_g$.
Let $k\geq 0$ be an integer.
The natural action of $\Gamma_g$ on $H$ gives a $\Gamma_g$-module structure on $\Hom_{\Z}(\wedge^{2k+1}H,\R/{\Z})$, where $\wedge^{2k+1}$ is the $(2k+1)$-th wedge product and the action on $\R/{\Z}$ is trivial.
This $\Gamma_g$-module gives rise to a local system $\mathcal{L}_k$ of free abelian groups on $\M_g$ by a pullback along $\pi$ on the universal family ${\mathcal C}_g$.
We consider two natural sections of $\mathcal{L}_k$.
The first is the well-known Abel-Jacobi map $A: {\mathcal C}_g\to \mathcal{L}_0$.
Let $X$ be a compact Riemann surface of genus $g\geq 2$ and $x\in X$ a base point.

The Abel-Jacobi map $A_{x}\colon X\to \Hom(H,\R/{\Z})$ is defined by
\[
A_{x}\colon X\ni y \mapsto \left(\omega\mapsto \int_{x}^{y}\omega
\right)\in \Hom(H,\R/{\Z}).
\]
The second is the \textit{pointed harmonic volume}
\[I: {\mathcal C}_g\to \mathcal{L}_1\]
studied by Harris \cite{0527.30032} and Pulte \cite{0678.14005}.
These two natural sections uniformly can be interpreted as an integration on the $(2k+1)$-th chain in the Jacobian variety of $X$.
See Section \ref{Abel-Jacobi maps and harmonic volume} for more details.
We focus on the latter.
We remark that a key object is the module $\Hom_{\Z}(\wedge^{3}H,\R/{\Z})$.

We define the pointed harmonic volume for a pointed compact Riemann surface $(X,x)$.
Let $K$ denote the kernel of $(\ , \ ) :H\otimes H \to \Z$ induced by the intersection pairing.
For $\sum_{i=1}^{n}a_i\otimes b_i\in K$
and $c\in H$,
the {\it pointed harmonic volume} is
defined to be a homomorphism $I_{(X,x)}\colon K\otimes H\to \R/{\Z}$
given by
\[I_{(X,x)}{\Biggl(}{\biggl(}\sum_{i=1}^{n}a_{i}\otimes b_{i}{\biggr)}\otimes c{\Biggr)}
=\sum_{i=1}^{n}\int_{\gamma}a_i b_i +\int_{\gamma}\eta
\quad \mathrm{mod} \ \mathbb{Z}.\]
Here the integral $\int_{\gamma}a_ib_i$ is Chen's iterated integral,
$\gamma$ is a loop in $X$ with base point $x$ whose
homology class is the Poincar\'e dual of $c$, and
$\eta$ is a unique smooth 1-form on $X$ satisfying the two conditions
$d\eta +\sum_{i=1}^{n} a_i\wedge b_i =0$
and
$\int_{X}\eta\wedge \ast\alpha =0$
for any closed 1-form $\alpha$ on $X$.
Through a natural homomorphism $K\otimes H\to \wedge^3 H$,
the pointed harmonic volume $I_{(X,x)}$ can be regarded as an element
of $\Hom_{\Z}(\wedge^3 H, \R/{\Z})$.

The pointed harmonic volume has been studied in the context of, for example,
algebraic geometry, number theory, complex analysis, and topology.
It is defined using Chen's \cite{0389.58001} iterated integrals, which were originally used in the context of differential geometry on loop spaces.
The pointed harmonic volume has the following three features.
First, as an extension of the Abel-Jacobi map,
the pointed harmonic volume captures the geometric information of pointed compact Riemann surfaces.
Actually, it induces the mixed Hodge structure of the truncation on the fundamental group ring of a pointed compact Riemann surface.
Using this, we obtain the so-called pointed Torelli theorem
stating that the pointed harmonic volume determines
the structure of the moduli space of pointed compact Riemann surfaces.
Second, the pointed harmonic volume is
explicitly computable for a number of special
compact Riemann surfaces including the Fermat curves.
In fact, there are not so many computable analytic invariants.
The pointed harmonic volume enables a quantitative study of the local structure of $\M_g$ and ${\mathcal C}_g$.
For example, we explain the nontriviality of the Ceresa cycle, see Section
\ref{The Chow group}, in the Jacobian varieties of compact Riemann surfaces.
Third, we can also compute the first variation of the pointed harmonic volume.
It is a twisted 1-form on ${\mathcal C}_g$ representing
the first extended Johnson homomorphism \cite{zbMATH00179521} on the mapping class group of a pointed oriented closed surface.
In other words, the first variation represents the first Morita-Mumford class or tautological class $e_1\in H^2(\M_g; \Z)$ \cite{ zbMATH03927904, zbMATH03882570}.
From the viewpoint of Hodge theory, Hain and Reed \cite{zbMATH05033769} obtained a similar result.
Since Kawazumi \cite{1170.30001} has reviewed the last of these, we make no mention of it here.
See Kawazumi's chapter \cite{1170.30001} for more details.
We remark that few values of the first variation of the harmonic volume
for a specific compact Riemann surface are known.
For yet another method for the harmonic volume,
we recommend the survey \cite{zbMATH02030179} by Hain.

In the following, we overview the theory of the harmonic volume.
More details will be given in subsequent sections.
Harris \cite{0527.30032} defined the harmonic volume
for compact Riemann surfaces using Chen's \cite{0389.58001}
iterated integrals of 1-forms of length 2.
See Section \ref{Iterated integrals} for the iterated integrals
and \ref{Harmonic volume} for the harmonic volume.
Prior to the harmonic volume,
Harris \cite{zbMATH03788720} studied an explicit formula of triple products of holomorphic cusp forms for $\mathrm{PSL}(2,\Z)$,
in which a prototype of the harmonic volume appears.
We define the harmonic volume.
Let $(H^{\otimes 3})^{\prime}$ be the kernel of
the natural homomorphism $p: H^{\otimes 3}
\ni a\otimes b\otimes c\mapsto ((a,b)c,(b,c)a,(c,a)b)\in H^{\oplus 3}$
induced by the intersection pairing on $H$.
It is a subgroup of $K\otimes H$.
The {\it harmonic volume} $I_X$ for $X$
is a restriction of the pointed harmonic volume $I_{(X,x)}$:
\[I_{X}=I_{(X,x)}|_{(H^{\otimes 3})^{\prime}}\colon
(H^{\otimes 3})^{\prime}\to \R/{\Z}.\]
It depends only on the complex structure of $X$,
and not on the choice of a Hermitian metric and base point.
The mapping class group $\Gamma_g$ acts naturally on $H$.
This induces the diagonal action of $\Gamma_g$ on
$\Hom_{\Z}((H^{\otimes 3})^{\prime}, \R/{\Z})$.
Let $\Aut X$ denote the group of biholomorphisms of $X$.
By construction, $I_X$ is $\Aut X$-invariant.
It can be regarded as a real analytic section of a local system on $\M_g$
obtained by the $\Gamma_g$-module $\Hom_{\Z}((H^{\otimes 3})^{\prime}, \R/{\Z})$.
The harmonic volume $2I_{X}$ for hyperelliptic curves is known to be trivial.
Nevertheless, the zero locus of the harmonic volume for nonhyperelliptic
curves is unknown.
Moreover the harmonic volume $I_X$ can be interpreted as the volume of
a 3-chain in the torus $\R^3/{\Z^3}$.
From this viewpoint, Faucette \cite{0783.14003}
extended it to the higher-dimensional harmonic volume.
See the end of Section \ref{Pointed Torelli Theorem}.

In Section \ref{The Chow group},
we introduce some basic concepts about algebraic cycles
on an algebraic variety, in particular, the Jacobian variety
$J(X)$ (or $J$) of $X$.
Weil \cite[pp. 331]{0168.18701} asked whether homologically trivial algebraic cycles are algebraically nontrivial in $J$.
That is, whether the Griffiths group is nontrivial, or
roughly speaking, whether these algebraic cycles can be ``continuously'' (algebraically)
deformed into the zero cycle.
By the Abel-Jacobi map $X\to J$, $X$ is embedded in $J$.
Its image of the $k$-th symmetric product of $X$ is denoted by $W_k$.
The algebraic $k$-cycle $W_k-W_k^{-}$ in $J$
is known to be homologous to zero.
We denote by $W_k^{-}$ the image of $W_k$
under multiplication by $-1$.
The cycle $W_k-W_k^{-}$ is called the $k$-th Ceresa cycle
for $1\leq k\leq g-1$.
It is one of the main subjects of this chapter.
If $X$ is hyperelliptic or $k=g-1$, then
the $k$-th Ceresa cycle is trivial.
Ceresa \cite{zbMATH03855282} proved that the $k$-th Ceresa cycle
is algebraically nontrivial in $J(X)$ for a generic compact Riemann surface
$X$ of genus $g\geq 3$ for $1\leq k\leq g-2$.

Section \ref{Abel-Jacobi maps and harmonic volume} establishes the relation between a higher Abel-Jacobi map and the harmonic volume.
In preparation, we briefly sketch the Hodge structure of a $\Z$-module of finite rank.
As the harmonic volume is the volume of a 3-chain in $\R^3/{\Z^3}$,
it is regarded as a point of an intermediate Jacobian of $J$.
Griffiths \cite{MR0260733} defined the $k$-th Abel-Jacobi map $\Phi_k$ from the
$k$-dimensional homologically trivial algebraic cycles on $J$ to the intermediate Jacobian of $J$ which is isomorphic to $\Hom_{\Z}(\wedge^{2k+1}H,\R/{\Z})$ here.
By integration on a 3-chain which bounds the first Ceresa cycle $W_1-W_1^{-}$,
the harmonic volume $2I_X$
can be identified with its image by the first Abel-Jacobi map of Griffiths,
{\it i.e.},
\[
2I_X=\Phi_1(W_1-W_1^{-})\in \Hom_{\Z}(\wedge^{3}H,\R/{\Z}).
\]
This identification induces a sufficient condition for the nontriviality of the Ceresa cycle in $J$ by computing some values of $2I_{X}$.

In Section \ref{Pointed Torelli Theorem},
we review the pointed Torelli theorem
as an application of the pointed harmonic volume.
Hain \cite{0654.14006}
defined a mixed Hodge structure on
the truncation of the fundamental group ring
of complex manifolds
by means of Chen's iterated integrals of 1-forms.
This is a different approach to that of Morgan \cite{zbMATH03621923}.
The classical Torelli theorem says that a compact Riemann surface $X$
is determined by its Jacobian variety $J(X)$ regarded as a complex torus.
Let $\pi_1(X, x)$ be the fundamental group of $X$ with base point $x$.
We define the augmentation ideal $J_{x}$ of the group ring $\Z\pi_1(X,x)$ by
the kernel of the map
$\Z\pi_1(X, x)\ni \sum a_{\gamma} \gamma \mapsto \sum a_{\gamma} \in \Z$.
The $k$-th power of the augmentation ideal is denoted by $J_{x}^{k}$.
From Chen's $\pi_1$ de Rham theorem, see Section \ref{Iterated integrals}, the pointed harmonic volume belongs to $\Hom_{\Z}(J_{x}/{J_{x}^{3}},\R)$.
The pointed Torelli theorem states that
the mixed Hodge structure on $\Hom_{\Z}(J_{x}/{J_{x}^{3}},\Z)$
determines the structure of ${\mathcal C}_g$.
Hain \cite{0654.14006} and Pulte \cite{0678.14005}
gave a proof using the pointed harmonic volume and an extension obtained by
a natural short exact sequence of mixed Hodge structures
\[
\SelectTips{cm}{}\xymatrix{
0 \ar[r] & H
\ar[r]& \Hom_{\Z}(J_{x}/{J_{x}^{3}},\Z)
\ar[r]& K
\ar[r] & 0.}
\]
In the proof of this theorem, the classical Torelli theorem
follows from the preservation of extensions of the mixed Hodge structures.
We obtain the biholomorphism of compact Riemann surfaces.
When we distinguish the base points,
the pointed harmonic volume plays an important role.
Section \ref{Pointed Torelli Theorem}
is devoted to this theorem.

In Section \ref{Nontrivial algebraic cycles in the Jacobian varieties},
we deal with an application of the harmonic volume, the nontriviality of the Ceresa cycle $W_k-W_k^{-}$ in the Jacobian variety $J(X)$.
The harmonic volume captures information about how $X$ is embedded in $J$.
From the nonvanishing of the first variation of the harmonic volume for special families of hyperelliptic curves, Harris \cite{0527.30032} obtained a result similar to Ceresa's.
The problem here is to give an explicitly compact Riemann surface satisfying the nontriviality condition.
Harris \cite{0523.14006} extended the harmonic volume with values
in $\C/{\Z[\sqrt{-1}]}$ and proved that
the first Ceresa cycle for the Fermat quartic
$F_4$ of genus 3
is algebraically nontrivial in $J(F_4)$.
Faucette \cite{0764.14015} extended this and
gave the proof that
the second Ceresa cycle for an unramified
double covering of $F_4$ is algebraically nontrivial.
Bloch \cite{0527.14008} further studied $F_4$
by means of $L$-functions.
There are a few explicit nontrivial examples apart from this.
In the following, we focus on nontriviality for the first Ceresa cycle.
The author proves in \cite{1222.14058}
nontriviality for the Klein quartic $K_4$ of genus 3
by a slight modification of this method.
The curve $K_4$ has a representation as a branched cyclic covering of $\C P^1$.
From this, we can compute certain values of the harmonic volume for $K_4$
by special values of the generalized hypergeometric function ${}_3F_2$.
See Section \ref{A generalized hypergeometric function} for ${}_3F_2$.
This method sheds some new light on nontriviality for the Ceresa cycles.
Moreover, he proves in \cite{1184.14018} nontriviality for the Fermat sextic
$F_6$ of genus 10 using this method, but applying it directly to other Fermat curves is difficult.
Let $\mathcal{O}$ denote the integer ring of the $N$-th cyclotomic field for a positive integer $N\geq 4$.
Otsubo \cite{1236.14009} extended the harmonic volume with values in
$(\mathcal{O}\otimes_{\Z}\R)/{\mathcal{O}}$.
He obtained its value and the algorithm in proving nontriviality for the $N$-th Fermat curves $F_{N}$ for any $N\geq 4$.
See Sections \ref{Fermat curves} and \ref{A generalized hypergeometric function}.
Moreover, he pointed out a relation between the harmonic volume and number theory using $L$-functions including nontorsionness of the higher Abel-Jacobi image $\Phi_{k}(W_k-W_k^{-})$ and the conjecture of Swinnerton-Dyer.
Using his method, we \cite{1009.0096} recently obtained nontriviality for some cyclic quotient $C_{N}$ of the Fermat curve.
Here $N$ is a prime number with $N=1$ modulo $3$.
See Section \ref{Cyclic quotients of Fermat curves}.
We remark that certain values of the harmonic volume for special
Riemann surfaces were calculated, but all the values are unknown.

\section{Preliminaries}
\subsection{Iterated integrals}
\label{Iterated integrals}
To define the harmonic volume for a compact Riemann surface,
we need to recall Chen's iterated integrals on a smooth manifold.
They have been developed
in various theories including the de Rham homotopy theory.
We introduce iterated integrals of 1-forms and
$\pi_1$ de Rham theorem which
states that the cohomology of the loop space of the smooth manifold can be
calculated by means of these integrals.
For a treatment of a more general case, we refer to
Chen \cite{0389.58001} and Hain \cite{0654.14006}.

Let $M$ be a smooth manifold and $\Omega^i(M)$ the smooth $i$-forms on $M$.
\begin{defn}
Denote $\gamma\colon [0,1]\to M$ as a piecewise smooth path.
For $\omega_1,\omega_2,\ldots,\omega_n\in \Omega^1(M)$,
we define an iterated (path) integral\index{iterated integral} of $\omega_1,\omega_2,\ldots,\omega_n$
along $\gamma$ by
\[
\int_{\gamma}\omega_1\omega_2\cdots \omega_n
=\idotsint_{0\leq t_1\leq t_2 \leq \cdots \leq t_n\leq 1}
f_1(t_1)f_2(t_2)\cdots f_n(t_n)dt_1dt_2\cdots dt_n,
\]
where $f_j(t)dt$ is the pullback $\gamma^{\ast}\omega_j$.
The number $n$ is called the length.
\end{defn}
The iterated integral is independent of the choice
of the parameterization of the path $\gamma$.
We introduce fundamental properties of iterated integrals.
Let $PM$ be the set of piecewise smooth paths in $M$.
For paths $\alpha,\beta\in PM$ such that
$\alpha(1)=\beta(0)$,
we denote their product by $\alpha\cdot \beta\in PM$ as usual.
It is easy to prove
\begin{lem}[\cite{0097.25803}]
\label{product of paths}
If $\omega_1,\omega_2,\ldots,\omega_n\in \Omega^1(M)$,
then we have
\[
\int_{\alpha\cdot\beta}\omega_1\omega_2\cdots \omega_n
=\sum_{i=0}^{r}\int_{\alpha}\omega_1\cdots \omega_i
\int_{\beta}\omega_{i+1}\cdots \omega_n.
\]
Here, we set $\int_{\gamma}\omega_1\omega_2\cdots \omega_n=1$
for $n=0$.
\end{lem}
\begin{exs}
\label{examples for product of paths}
\begin{align*}
\int_{\alpha\cdot \beta}\omega_1\omega_2
&=\int_{\alpha}\omega_1\omega_2+\int_{\alpha}\omega_1\int_{\beta}\omega_2
+\int_{\beta}\omega_1\omega_2,\\
\int_{\alpha\cdot \beta\cdot \gamma}\omega_1\omega_2
&=\left(\int_{\alpha}+\int_{\beta}+\int_{\gamma}\right)\omega_1\omega_2
+\int_{\alpha}\omega_1\int_{\beta}\omega_2
+\int_{\alpha}\omega_1\int_{\gamma}\omega_2
+\int_{\beta}\omega_1\int_{\gamma}\omega_2.
\end{align*}
\end{exs}

A permutation $\sigma$ of $\{1,2,\ldots, r+s\}$
is a {\it shuffle} of type $(r,s)$ if
\[
\sigma^{-1}(1)<\sigma^{-1}(2)<\cdots <\sigma^{-1}(r)
\]
and
\[
\sigma^{-1}(r+1)<\sigma^{-1}(r+2)<\cdots <\sigma^{-1}(r+s).
\]
The following formula was derived by Ree \cite{zbMATH03136708}.
\begin{prop}[Shuffle relation]
\label{Shuffle relation}
Let $\omega_1,\omega_2,\ldots,\omega_{r+s}\in \Omega^1(M)$.
Then we have
\[
\int_{\alpha}\omega_1\omega_2\cdots \omega_r
\int_{\alpha}\omega_{r+1}\omega_{r+2}\cdots \omega_{r+s}
=\sum_{\sigma}\int_{\alpha}\omega_{\sigma(1)}\omega_{\sigma(2)}
\cdots \omega_{\sigma(r+s)},
\]
where $\sigma$ runs over the shuffles of type $(r,s)$.
\end{prop}
\begin{exs}
\begin{align*}
\int_{\alpha}\omega_1\int_{\alpha}\omega_2
&=\int_{\alpha}\omega_1\omega_2+\int_{\alpha}\omega_2\omega_1,\\
\int_{\alpha}\omega_1\int_{\alpha}\omega_2\omega_3
&=\int_{\alpha}\omega_1\omega_2\omega_3+\int_{\alpha}\omega_2\omega_3\omega_1
+\int_{\alpha}\omega_2\omega_1\omega_3.
\end{align*}
\end{exs}

A function $F\colon PM\to \R^{m}$ is a {\it homotopy functional}\index{homotopy functional} if
for every $\gamma\in PM$ $F(\gamma)$ depends only on the homotopy class of $\gamma$
relative to its endpoints.
For each $x\in M$, a homotopy functional $F$ induces a function
$\pi_1(M, x)\to \R$.
It is well-known that a function
\[
\int\omega\colon PM\ni \gamma\mapsto \int_{\gamma}\omega\in \R
\]
is a homotopy functional if and only if $\omega\in \Omega^{1}(M)$ is closed.
In general, even if $\omega_1,\omega_2\in \Omega^{1}(M)$ are closed, a function
\[
\int\omega_1\omega_2\colon PM\ni \gamma\mapsto
\int_{\gamma}\omega_1\omega_2\in \R
\]
may not be a homotopy functional.
For iterated integrals of length 2,
we need a correction term $\eta\in \Omega^1(M)$.
\begin{prop}
\label{homotopy functionals}
Let $\omega_1,\omega_2,\ldots,\omega_r, \eta\in \Omega^1(M)$
and $c_{ij}\in \R$ for $1\leq i,j\leq r$.
Suppose that each $\omega_i$ is closed.
Then
\[
\sum c_{ij}\int \omega_i\omega_j +\int \eta
\]
is a homotopy functional if and only if
\[
d\eta +\sum c_{ij}\omega_i\wedge \omega_j =0.
\]
\end{prop}
This can be verified by lifting the integral to the universal covering of $M$.
Let $B_{s}(M)$ be the vector space of finite linear combinations
of iterated integrals on $M$ of length $\leq s$.
Each elements of $B_{s}(M)$ is represented by
\[
I=\lambda +\sum a_i\int \omega_i
+\sum a_{ij}\int \omega_i \omega_j
+\cdots +\sum_{|J|=s}a_{J}\int \omega_{j_1}\omega_{j_2}\cdots \omega_{j_s}.
\]
For the group ring $\Z\pi_1(X, x)$,
we define the augmentation ideal\index{augmentation ideal} $J_{x}$ by
the kernel of the map
$\Z\pi_1(X, x)\ni \sum a_{\gamma} \gamma \mapsto \sum a_{\gamma} \in \Z$.
It is clear that $J_{x}$ is generated by $\gamma -1$ for any loop
$\gamma$ with base point $x$.
Here $1$ denotes the constant path at $x$.
The $k$-th power of the augmentation ideal is denoted by $J_{x}^{k}$.
For $I=\int \omega\in B_1(M)$ and $(\alpha-1)(\beta-1)\in J_{x}^{2}$,
we have immediately
\[
\langle I, (\alpha-1)(\beta-1)\rangle =0,
\]
where $\alpha, \beta$ are loops with base point $x$.
Moreover, Examples \ref{examples for product of paths} imply
\[
\langle I, (\alpha-1)(\beta-1)(\gamma-1)\rangle =0
\]
for $I\in B_2(M)$ and $(\alpha-1)(\beta-1)(\gamma-1)\in J_{x}^{3}$.
We generalize these expressions.
\begin{lem}
\label{values of augmentation ideals}
Let $\alpha_1, \alpha_2, \ldots, \alpha_s$ denote loops with base point $x$.
If $I\in B_{r}(M)$ and $r<s$, then we have
\[
\langle I, (\alpha_1-1)(\alpha_2-1)\cdots(\alpha_s-1) \rangle =0.
\]
\end{lem}
The proof of this lemma is straightforward.
Let $\omega_1,\omega_2, \ldots,\omega_s \in \Omega^1(M)$
and $(\alpha_1-1)(\alpha_2-1)\cdots(\alpha_s-1)$ as above.
Lemma \ref{product of paths} implies
\begin{equation}
\left\langle \int \omega_1\omega_2\cdots \omega_s,
(\alpha_1-1)(\alpha_2-1)\cdots(\alpha_s-1) \right\rangle =
\int_{\alpha_1}\omega_1\int_{\alpha_2}\omega_2\cdots
\int_{\alpha_s}\omega_s.
\end{equation}

The set of piecewise smooth loops in $M$
with base point $x\in M$ is denoted by $P_{x}M$.
Let $H^{0}(B_{s}(M), x)$ be the set of
homotopy functionals $P_{x}M\to \R$.
Integration induces a linear map
\[H^{0}(B_{s}(M), x)\ni I\mapsto \left(\gamma \mapsto
\langle I, \gamma \rangle\right)
\in\Hom_{\Z}(\Z \pi_1(X, x),\R).\]
Lemma \ref{values of augmentation ideals}
gives us $I|_{J_{x}^{s+1}}=0$ for $I\in B_{s}(M)$.
Chen proved the following $\pi_1$ de Rham theorem\index{$\pi_1$ de Rham theorem} \cite[Theorem 5.3]{0301.58006}.
\begin{thm}
The integration map induces an isomorphism
\[
H^0(B_s(M), x)\to \Hom_{\Z}(\Z \pi_1(X, x)/{J_{x}^{s+1}},\R).
\]
\end{thm}
Let $\overline{B}_{s}(M)$ denote the subset of elements of $B_{s}(M)$
whose constant term vanishes.
A similar result follows.
\begin{cor}
\label{de Rham theorem}
We have an isomorphism
\[
H^0(\overline{B}_{s}(M), x)\to \Hom_{\Z}(J_{x}/{J_{x}^{s+1}},\R).
\]
\end{cor}
If we replace $\R$ with $\C$, similar results are obtained.
The pointed harmonic volume $I_{(X,x)}$
for a pointed compact Riemann surface $(X,x)$
can be interpreted as an element of
$H^0(\overline{B}_{2}(X), x)$.

\subsection{Harmonic volume}
\label{Harmonic volume}
We review the definition of the harmonic volume for a
compact Riemann surface and its properties.
It is a complex analytic invariant defined by Chen's iterated integrals of length $2$.
It also gives information about how the compact Riemann surface is embedded in its Jacobian variety.
First, we define the pointed harmonic volume for a pointed compact Riemann surface.

Let $X$ be a compact Riemann surface or smooth projective curve over $\C$ of genus $g\geq 2$.
(See Farkas and Kra \cite{0764.30001} for an introduction to Riemann surfaces.)
The surface $X$ is homeomorphic to an oriented closed surface $\Sigma_g$ of genus $g$.
Its mapping class group\index{mapping class group}, denoted by $\Gamma_g$, is the group of isotopy classes of orientation-preserving diffeomorphisms of $\Sigma_g$.
The group $\Gamma_g$ acts naturally on the first integral homology group
$H_1(X; \Z)=H_1(\Sigma_g; \Z)$.
Let $H$ denote the first integral cohomology group $H^1(X; \Z)$.
By Poincar\'e duality, $H$ is isomorphic to $H_1(X; \Z)$ as
$\Gamma_g$-modules.
The Hodge star operator $\ast$ is locally given by
$\ast (f_1(z)dz + f_2(z)d\bar{z})=-\sqrt{-1}\,f_1(z)dz + \sqrt{-1}\,f_2(z)d\bar{z}$ in a local coordinate $z$.
It depends only on the complex structure and not on the choice of a Hermitian metric.
The real Hodge star operator $\ast\colon \Omega^1(X)\to \Omega^1(X)$
is given by restriction.
Using the Hodge theorem, we identify $H$ with the space of real harmonic 1-forms on $X$ with $\Z$-periods, {\it i.e.},
$H=\{\omega\in \Omega^1(X);\, d\omega=d\ast \omega=0,
\ \int_{\gamma}\omega\in \Z \text{ for any loop }\gamma\}.$
We introduce the following lemma for homotopy functionals.
See Proposition \ref{homotopy functionals}.
\begin{lem}
Let $\omega_i\in \Omega^1(X)$ and $c_{ij}\in \R$.
If $\sum \int_{X}c_{ij}\omega_i\wedge \omega_j =0$,
then there exists a 1-form $\eta\in \Omega^1(X)$
such that $d\eta +\sum c_{ij}\omega_i\wedge \omega_j =0$.
\end{lem}
Indeed, an exact sequence
\[
\SelectTips{cm}{}\xymatrix{
0\ar[r] & \C \ar[r] & \Omega^0(X) \ar[r]^{d\ast d}
& \Omega^2(X) \ar[r]^{\int_{X}} & \C \ar[r] & 0
}
\]
gives a function $h\in \Omega^0(X)$ such that
$d\ast d\, h=\sum c_{ij} \omega_i\wedge \omega_j$.
Here the left vector space $\C$ means the constant functions.
Put $\eta=-\ast d h$.

Let $x\in X$ be a point.
We define the pointed harmonic volume for $(X,x)$ in the following way.
Let $K$ be the kernel of $(\ , \ ) :H\otimes H \to \Z$ induced by the intersection pairing.
For a given $\sum_{i=1}^{n}a_i\otimes b_i\in K$,
there exists an $\eta\in \Omega^1(X)$ satisfying the two conditions
\[d\eta +\sum_{i=1}^{n} a_i\wedge b_i =0\]
and
\[
\int_{X}\eta\wedge \ast\alpha =0
\]
for any closed 1-form $\alpha \in \Omega^1(X)$.
The second condition determines $\eta$ uniquely.
We can choose $\eta=-\ast d h$ with
$d\ast d\, h=\sum_{i=1}^{n}a_i\wedge b_i$
as in the above lemma.
This $\eta$ readily satisfies the above two conditions.
For any pointed compact Riemann surface $(X,x)$, the homotopy functional
\[
\varphi\colon P_{x}X\ni \gamma \mapsto
\sum_{i=1}^{n}\int_{\gamma}a_i b_i +\int_{\gamma}\eta \in\R
\]
induces a map
\begin{equation}
\label{homotopy functional}
\overline{\varphi}\colon \Z \pi_{1}(X, x)\to \R.
\end{equation}
We remark that this $\overline{\varphi}$ is an element of
$H^0(\overline{B}_{2}(X), x)\cong \Hom_{\Z}(J_{x}/{J_{x}^{3}},\R)$.
Lemma \ref{product of paths} yields
\[
\varphi(\alpha\beta)
=\varphi(\alpha)
+\varphi(\beta)
+\sum_{i=1}^{n}\int_{\alpha}a_i \int_{\beta}b_i
\]
for loops $\alpha, \beta\in P_{x}X$.
From the assumption, $\int_{\alpha}a_i, \int_{\beta}b_i\in \Z$ for each $i$.
Using a natural projection $\R\to \R/{\Z}$, the map $\overline{\varphi}\colon \Z \pi_{1}(X, x)\to \R/{\Z}$ is a homomorphism.
Furthermore, this gives the homomorphism
\[
\overline{\varphi}\colon H_{1}(X;\Z)\to \R/{\Z}.
\]
We define the notion of pointed harmonic volume \cite{0678.14005}.
\begin{defn}
For $\sum_{i=1}^{n}a_i\otimes b_i\in K$ and $c\in H$,
the pointed harmonic volume\index{pointed harmonic volume}\index{harmonic volume!pointed} is a homomorphism $K\otimes H\to \R/{\Z}$
\[I_{(X,x)}{\Biggl(}{\biggl(}\sum_{i=1}^{n}a_{i}\otimes b_{i}{\biggr)}\otimes c{\Biggr)}=\overline{\varphi}(c)
\quad \mathrm{mod} \ \mathbb{Z}.\]
Here
$\overline{\varphi}$ is defined in the way stated above
and the homology class $c$ is considered as
a loop in $X$ with base point $x$.
\end{defn}
\begin{rem}
\label{transposition}
From Proposition \ref{Shuffle relation}, we have
\[
I_{(X,x)}{\Biggl(}{\biggl(}\sum_{i=1}^{n}a_{i}\otimes b_{i}{\biggr)}\otimes c{\Biggr)}
=-I_{(X,x)}{\Biggl(}{\biggl(}\sum_{i=1}^{n}b_{i}\otimes a_{i}{\biggr)}\otimes c{\Biggr)}
\quad \mathrm{mod} \ \mathbb{Z}.
\]
\end{rem}
Harris \cite{0527.30032} gave the same definition of $I_{(X,x)}$, and called the pointed harmonic volume by Pulte.
The pointed harmonic volume $I_{(X,x)}$ is naturally regarded
as an element of $\Hom_{\Z}(\wedge^3 H, \R/{\Z})$
and a section of the local system $\mathcal{L}_1$ on ${\mathcal C}_g$ defined in
Section \ref{Introduction}.

The harmonic volume is a restriction of the pointed harmonic volume $I_{(X,x)}$.
We denote by $(H^{\otimes 3})^{\prime}$ the kernel of
the natural homomorphism
$p\colon H^{\otimes 3} \to H^{\oplus 3}$
defined by
$p(a\otimes b\otimes c)=((a, b)c, (b, c)a, (c, a)b)$.
The group $K\otimes H$ is a subgroup of $(H^{\otimes 3})^{\prime}$.
We have a natural short exact sequence
\[
\SelectTips{cm}{}\xymatrix{
0\ar[r] &(H^{\otimes 3})^{\prime} \ar[r]&
H^{\otimes 3} \ar[r]^{p}& H^{\oplus 3} \ar[r] & 0.
}
\]
The rank of the free $\Z$-module $(H^{\otimes 3})^{\prime}$
is $(2g)^3-6g$.
\begin{defn}[\cite{0527.30032}]
The {\it harmonic volume}\index{harmonic volume} $I_X$ for $X$
is a linear form on $(H^{\otimes 3})^{\prime}$
with values in $\R/{\Z}$ defined by the restriction of
$I_{(X,x)}$ to $(H^{\otimes 3})^{\prime}$, {\it i.e.},
\[I_{X}=I_{(X,x)}|_{(H^{\otimes 3})^{\prime}}\colon
(H^{\otimes 3})^{\prime}\to \R/{\Z}.\]
\end{defn}
From Lemma \ref{product of paths},
the harmonic volume $I_X$ is independent of the choice of base point $x$.
Let $S_3$ be the third symmetric group.
We explain a cyclic invariance of $I_X$ by the natural action of $S_3$ on $(H^{\otimes 3})^{\prime}$.
Combining Stokes' theorem and Remark \ref{transposition}, we have
\begin{equation}
\label{invariant by the action of permutation}
I_X\left(\sum_i \omega_{\sigma(1), i}\otimes \omega_{\sigma(2), i}\otimes \omega_{\sigma(3), i}\right)
=\mathrm{sgn}(\sigma)I_X\left(\sum_i \omega_{1, i}\otimes \omega_{2, i}\otimes \omega_{3, i}
\right)\
\mathrm{mod}\ \Z,
\end{equation}
where $\sum_i \omega_{1, i}\otimes \omega_{2, i}\otimes \omega_{3, i}
\in (H^{\otimes 3})^{\prime}$ and $\sigma$ is an element of $S_3$.

We present examples of calculation of the harmonic volume for hyperelliptic curves.
The hyperelliptic curve $C$ is the compactification of the plane curve in the $(z,w)$ plane
$\C^2$
\[w^2=\prod_{i=0}^{2g+1} (z-p_i),\]
where $p_0, p_1,\ldots, p_{2g+1}$ are some distinct points on $\C$.
It admits the hyperelliptic involution given by $(z,w)\mapsto (z,-w)$.
Let $\pi$ be the $2$-sheeted covering
$C \to \C P^1, (z,w) \mapsto z$, branched over $2g+2$
branch points $\{p_i\}_{i=0,1,\cdots ,2g+1}$
and $P_i \in C$ a ramification point such that $\pi(P_i)=p_i$.
It is known that $\{P_i\}_{i=0,1,\ldots,2g+1}$ is just the set of
all the Weierstrass points on any hyperelliptic curve $C$.
We outline the computation of the harmonic volumes for hyperelliptic curves.
See \cite{1090.14007,1133.14030} for details.
For any hyperelliptic curve $C$,
one has $I_{C}=0$ or $1/{2}\ \textnormal{mod}\ \Z$ by the existence of the hyperelliptic involution.
The computation was performed using a suitable choice of symplectic basis $\{x_i,y_i\}_{i=1,\ldots,g}$ of $H$ (\cite[p.800]{1090.14007}).
For example,
\begin{center}
$I_C((x_i\otimes y_i-x_{k+1}\otimes y_{k+1})\otimes y_k)=
\left\{
 \begin{array}{cc}
  \dfrac{1}{2} & (i<k, 2\leq k\leq g-1),\\
  0 & \text{otherwise}.
 \end{array}
\right.$
\end{center}
We have two ways to compute the harmonic volumes for all the hyperelliptic curves.
First the computation can be reduced to that of a single hyperelliptic curve.
Second we use basic results from the cohomology group of
the hyperelliptic mapping class group
which is composed of the centralizer of
the hyperelliptic involution in $\Gamma_g$.
Similarly the author \cite{1133.14030} obtained the pointed harmonic volume
$I_{(C,P_j)}$ for the Weierstrass pointed hyperelliptic curves.
However, the pointed harmonic volume for other pointed hyperelliptic curves
has yet to be determined.

We may consider the harmonic volume as an element of
$\Hom_{\Z}((\wedge^3 H)^{\prime},\R/{\Z})$.
Let $j_2\colon H^{\otimes 3}\to \wedge^3 H$ be
a natural homomorphism
\[j_2(a\otimes b\otimes c)= a\wedge b\wedge c,\]
where $\wedge^3 H$ denotes the third exterior product of $H$.
We have a homomorphism of short exact sequences
\[
\SelectTips{cm}{}\xymatrix{
0\ar[r] &(H^{\otimes 3})^{\prime} \ar[r] \ar[d]_{j_1}&
H^{\otimes 3} \ar[d]_{j_2}\ar[r]^{p}& H^{\oplus 3} \ar[d]_{j_3}\ar[r] & 0 \\
0\ar[r] &(\wedge^3 H)^{\prime} \ar[r] & \wedge^3 H \ar[r]^{\bar{p}}&
H \ar[r] & 0 ,
}
\]
where $j_3(a,b,c)=a+b+c$, $\bar{p}(a\wedge b\wedge c)=(a, b)c+(b, c)a+(c, a)b$
and $j_1$ is the restriction homomorphism of $j_2$ to $(H^{\otimes 3})^{\prime}$.
The rank of the free $\Z$-module $(\wedge^3 H)^{\prime}$ is ${2g \choose 3}-2g$.
Using Formula (\ref{invariant by the action of permutation}),
it is easy to show the following:
\begin{prop}
We can take a homomorphism $\nu_{X}\in \Hom_{\Z}((\wedge^3 H)^{\prime},\R/{\Z})$
satisfying the commutative diagram
\[
\SelectTips{cm}{}\xymatrix{
(H^{\otimes 3})^{\prime}\ar[r]^{2I_X} \ar[d]_{j_1}
& \R/{\Z}\\
(\wedge^3 H)^{\prime} \ar[ur]_{\nu_X}.&
}
\]
\end{prop}
The homomorphism $\nu_X$ is also called the harmonic volume.
Another description of $\nu_X$ is as follows.
For simplicity, we consider
$\omega_1\wedge \omega_2 \wedge \omega_3\in (\wedge^3 H)^{\prime}$,
and fix a base point $x\in X$.
The map $A_1\colon X\to \R^3/{\Z^3}$, a kind of Abel-Jacobi map,
is defined by:
\[
A_1(y)
=\left(\int_{x}^{y}\omega_{1},\int_{x}^{y}\omega_{2},
\int_{x}^{y}\omega_{3}\right)\ \mathrm{mod}\ \Z^3.
\]
There exists a 3-chain $c_3$ in $\R^3/{\Z^3}$
such that the image $A_1(X)=\partial c_3$ modulo integral 2-chains.
For each $k=1,2,3$, we can take the coordinates $x_{k}$ on $\R^3/{\Z^3}$
such that $A_1^{\ast}(dx_{k})=\omega_{k}$.
The volume of $c_3$ modulo $\Z$ is independent of the choice of $c_3$
and base point $x$.
The following proposition is proved by a straightforward computation.
\begin{prop}
\label{another definition of harmonic volume}
We have
\[
\nu_{X}\left(\omega_{1}\wedge \omega_{2}\wedge \omega_{3}
\right)
=2\int_{c_3}dx_{1}\wedge dx_{2} \wedge dx_{3}
\quad \mathrm{mod}\ \Z.
\]
\end{prop}

\subsection{The Chow group}
\label{The Chow group}
The harmonic volume can be applied to the nontriviality of the Ceresa cycles
in the Jacobian varieties of compact Riemann surfaces.
From this complex analytic method, the harmonic volume can capture more detailed information of these cycles.
The Ceresa cycle is an element of the Chow group that, to begin, we need to define.
A general reference here is Fulton \cite{0885.14002}.

Let $V$ be a smooth projective variety over $\C$.
The group $\mathcal{Z}_{k}(V)$ is defined to be
the free abelian group generated by the irreducible subvarieties
$W$ on $V$ of dimension $k$.
It is called the group of algebraic cycles\index{algebraic cycles} of dimension $k$.
We consider the algebraic equivalence on $V$.
The algebraic cycle $Z\in \mathcal{Z}_{k}(V)$
is algebraically equivalent
to zero\index{algebraically equivalent
to zero} if there exists a smooth curve $C$,
a cycle $T\in \mathcal{Z}_{k}(V\times C)$,
and two points $x_1,x_2\in C$ such that
$Z=i_1^{\ast}(T)-i_2^{\ast}(T)$,
where $i_j\colon V\hookrightarrow V\times C$ is $i_j(x)=
(x,x_j)$ for $j=1,2$.
For $C=\C P^1$, we call it rationally equivalent.
Denote
\[
\mathcal{Z}_{k}(V)_{\textrm{alg}}
=\{Z\in \mathcal{Z}_{k}(V); Z \text{ is algebraically equivalent to }0\}
\]
and $\mathcal{Z}_{k}(V)_{\textrm{rat}}$ in the similar way.
We define the Chow group\index{Chow group} of dimension $k$ by
\[
\CH_k(V):= \mathcal{Z}_{k}(V)/{\mathcal{Z}_{k}(V)_{\textrm{rat}}}.
\]
Set $\CH_k(V)_{\textrm{alg}}
:= \mathcal{Z}_{k}(V)_{\textrm{alg}}/{\mathcal{Z}_{k}(V)_{\textrm{rat}}}$
which is the group of algebraically trivial cycles modulo rational equivalence.
The cycle class map $\CH_k(V)\to H_{2k}(V;\Z)$ is obtained by
linearly extending the map to a subvariety $i\colon W\hookrightarrow V$
associated with its homology class $i_{\ast}[W]$.
The kernel of this map is denoted by
$\CH_k(V)_{\textrm{hom}}$.
It is the group of homologically trivial cycles modulo rational equivalence.
We have known inclusions
\[
\CH_k(V)_{\textrm{alg}}\subset \CH_k(V)_{\textrm{hom}}\subset \CH_k(V).
\]

The problem is to determine whether a cycle homologous to zero in $V$ is algebraically equivalent to zero.
If the cycle homologous to zero is not, then we find the Griffiths group of $V$ is nontrivial.
Here the $k$-th Griffith group\index{Griffith group} $\mathrm{Griff}_{k}(V)$ is defined by
$\mathrm{Griff}_{k}(V):=\CH_{k}(V)_{\textrm{hom}}/{\CH_{k}(V)_{\textrm{alg}}}$.
In particular, we are interested in this group for varieties and cycles defined over $\Z$.
See \cite{0897.14004} for example.

In this subsection and Sections \ref{Abel-Jacobi maps and harmonic volume} and \ref{Pointed Torelli Theorem}, we use the notation $H^1=H^1(X; \Z)$ and $H_1=H_1(X; \Z)$.
Let $J(X)$ (or $J$) be the Jacobian variety\index{Jacobian variety} of the compact Riemann surface $X$
of genus $g\geq 3$.
We denote by $H^{1,0}$ the complex vector space of holomorphic 1-forms on $X$,
and fix a base point $x\in X$.
The Abel-Jacobi map\index{Abel-Jacobi map} $X\to J$ is defined by
\[
A_x\colon X\ni y\mapsto
\left(\omega\mapsto \int_{x}^{y}\omega \right)
\in J={(H^{1,0})}\spcheck/{H_1},
\]
where $\spcheck$ means the complex linear dual.
We define an inclusion map by $H_{1}\ni \gamma \mapsto \left(\omega\mapsto \int_{\gamma}\omega\right)\in {(H^{1,0})}\spcheck$.
If we fix $\omega_1,\omega_2,\ldots,\omega_g$ as a basis of $H^{1,0}$, then $J$ is a $g$-dimensional complex torus obtained as the quotient of $\C^g$ by an abelian group.
Since $g$ is positive, $X$ can be embedded into $J$.
We may identify $H^{1}$ with $H^1(J ;\Z)$
and $\wedge^k H^1$ with $H^{k}(J; \Z)$ for $1\leq k\leq 2g$.
Let $X^k$ denote the $k$-fold product of $X$
and let $X_k$ denote the $k$-th symmetric product.
We have a commutative diagram
\begin{equation}
\label{AJ commutative diagram}
\SelectTips{cm}{}\xymatrix{
 X^k \ar[r]^{(A_x)^{k}} \ar[d]& J^k \ar[d]\\
 X_k \ar[r]^{(A_x)_{k}} & J,
}
\end{equation}
where the left-hand side map is the natural projection,
the right-hand an addition,
and $(A_x)_{k}$ is the induced homomorphism.
With an abuse of notation, we also denote the latter by $A_x$;
its image is denoted by $W_k(x):=A_x(X_{k})$.
For $x,y\in X$,
the algebraic $k$-cycle $W_k(x)-W_k^{\pm}(y)$ in $J$
is called the $k$-th Ceresa cycle\index{Ceresa cycle}.
Here we denote by $W_k^{-}(y)$ the image of $W_k(y)$
under multiplication by $-1$ and $W_k^{+}(y)=W_k(y)$.
The multiplication induces the identity map on $H_{2k}(J; \Z)$.
The $k$-th Ceresa cycle is homologous to zero,
{\it i.e.}, $W_k(x)-W_k^{\pm}(y)\in \CH_k(J)_{\textrm{hom}}$.
Since $W_k(x)-W_k(y)\in \CH_k(J)_{\textrm{alg}}$,
the class of $W_k(x)-W_k^{-}(y)$ modulo algebraic equivalence
does not depend on $x$ and $y$.
We omit the base points $x,y$, unless otherwise stated.

If $X$ is hyperelliptic, then $W_k-W_k^{-}\in \CH_k(J)_{\textrm{alg}}$.
Indeed, we may choose a Weierstrass point $x\in X$ as base point.
The hyperelliptic curve $X$ has the hyperelliptic involution $\iota$ which is a
biholomorphism of $X$ of order $2$ and fixes all the Weierstrass points in $X$.
Since the action of $\iota$ on $H_{1}$ is multiplication by $-1$,
the multiplication by $-1$ on $J(X)$ restricts to $\iota$ on $X$.
Then we have $W_k-W_k^{-}\in \CH_k(J)_{\textrm{alg}}$.
If $k=g-1$, then the cycle $W_{g-1}-W_{g-1}^{-}$ is known to be trivial.
Ceresa's theorem \index{Ceresa's theorem} \cite{zbMATH03855282} implies that
$\mathrm{Griff}_{k}(J(X))\neq 0$
for a generic (nonhyperelliptic) curve $X$ of genus $g\geq 3$ for $1\leq k\leq g-2$.
In Section \ref{Nontrivial algebraic cycles in the Jacobian varieties},
we give explicit $X$'s such that $\mathrm{Griff}_{k}(J(X))\neq 0$.

\section{Abel-Jacobi maps and harmonic volume}
\label{Abel-Jacobi maps and harmonic volume}
To show a relation between the harmonic volume
$\nu_{X}=2I_{X}$ and the image of the Abel-Jacobi map of Griffiths,
we begin to recall the definition of
the Hodge structure on a $\Z$-module of finite rank
and an intermediate Jacobian of $V$.
We give a sufficient condition for
the Ceresa cycle of $\CH_1(J(X))_{\textrm{hom}}$
to be algebraically nontrivial.

A {\it Hodge structure}\index{Hodge structure} of weight-$w$ on a $\Z$-module $H_{\Z}$
of finite rank is a direct sum decomposition
\begin{center}
$\displaystyle H_{\C}:=H_{\Z}\otimes_{\Z}\C=\bigoplus_{p+q=w}H^{p,q}$
with $H^{p,q}=\overline{H^{q,p}}$.
\end{center}
The Hodge filtration associated to this Hodge structure is given by
\[
F^{p}H_{\C}=\bigoplus_{r\geq p}H^{r,s}.
\]
We immediately obtain
$H^{p,q}=F^{p}H_{\C}\cap \overline{F^{q}H_{\C}}$
and the decreasing filtration
\[
H_{\C}\supset \cdots \supset F^{p}H_{\C}\supset F^{p+1}H_{\C}\supset \cdots.
\]
Conversely, a decreasing filtration $F^{p}$ of $H_{\C}$ with the condition
$F^{p}\cap \overline{F^{q}}=0$ whenever $p+q=k+1$ determines a weight-$k$
Hodge structure by putting
\[H^{p,q}=F^{p}\cap \overline{F^{q}}.\]
If $V$ is a compact K\"{a}hler manifold,
the cohomology group $H^{w}(V; \Z)$ underlies a Hodge structure of weight-$w$.
Here, $H^{p,q}$ is the space of cohomology classes whose harmonic representative is of type $(p,q)$.
For two Hodge structures $A$ and $B$ of weight $w$ and $v$ respectively,
we have Hodge structures $A\otimes B$ and $\Hom(A,B)$
of weight $w+v$ and $v-w$ respectively:
\[
(A\otimes B)^{p,q}=\bigoplus_{i,j} A^{i,j}\otimes B^{p-i,q-j}
\]
and
\[
\Hom(A,B)^{p,q}=\{f\colon A_{\C}\to B_{\C}; f(A^{i,j})\subset B^{i+p,j+q}\}.
\]
Here $A_{\C}=\oplus_{p+q=w}A^{p,q}$ and $B_{\C}=\oplus_{p+q=v}B^{p,q}$.
If $V$ is a compact K\"{a}hler manifold,
the homology group $H_{w}(V; \Z)$ carries a natural Hodge structure of weight $-k$.
Indeed, we recall the isomorphism $H_{k}(V;\C)\cong \Hom_{\C}(H^{k}(V; \C),\C)$
and obtain the direct sum decomposition
\[
H_{w}(V; \C)=\bigoplus_{-p-q=-w}H_{w}(V; \C)^{-p,-q},
\]
where $
H_{w}(V; \C)^{-p,-q}
=\{
f\colon H^{w}(V; \C)\to \C\, ; f(H^{r,s})=0 \text{ whenever }(r,s)\neq (p,q)
\}.
$

If $H_{\Z}$ has a Hodge structure of odd weight $2k+1$, we define a complex torus by
\[
J(H):=H_{\C}/{(F^{k+1}H_{\C}+H_{\Z})}=\overline{F^{k+1}H_{\C}}/{H_{\Z}}.
\]
The real torus $H_{\R}/{H_{\Z}}$ is denoted by $J_{\R}H$,
where $H_{\R}$ denotes $H_{\Z}\otimes_{\Z}\R$.
The inclusion $H_{\R}\to H_{\C}$ induces an isomorphism of real Lie
groups $J_{\R}H\to J(H)$.
A natural projection $\R\to \R/{\Z}$ induces an isomorphism of real tori
\[
J_{\R}\Hom(H,\C) \cong \Hom_{\Z}(H_{\Z}, \R/{\Z}).
\]
\begin{lem}
\label{isomorphism of real tori}
We have a natural isomorphism of real tori
\[
J\Hom(H,\C) \to \Hom_{\Z}(H_{\Z}, \R/{\Z}).
\]
\end{lem}
We focus on the homology group $H_{2k+1}(V ;\Z)$ with Hodge structure of weight $-2k-1$.
The $k$-th {\it intermediate Jacobian}\index{intermediate Jacobian} of Griffiths is defined by
\[
J_k(V):=J(H_{2k+1}(V ;\Z))
= \dfrac{F^{-k}H_{2k+1}(V; \C)}{H_{2k+1}(V ;\Z)}
= \dfrac{{(F^{k+1}H^{2k+1}(V; \C))}\spcheck}{H_{2k+1}(V ;\Z)},
\]
The $J_{0}(X)$ is the Jacobian variety $J=J(X)$.

For an element $Z\in\CH_k(V)_{\textrm{hom}}$,
we can take a topological $(2k+1)$-chain $W$ so that
$Z=\partial W$.
The integration
$H^{2k+1}(V ;\C)\ni \omega\mapsto
\int_{W}\omega \in \C$
induces the Abel-Jacobi map of Griffiths\index{Abel-Jacobi map!of Griffiths}
\[\Phi_{k}\colon
\CH_k(V)_{\textrm{hom}}\to J_k(V),\]
where $\omega$ is a harmonic $(2k+1)$-form
on $V$ with integral periods \cite[Section 4]{0678.14005}.
We give a key lemma in proving nontriviality for the Ceresa cycles.
\begin{lem}
\label{vanishing condition}
Let $F^{p}$ denote $F^{p}H^{2k+1}(V; \C)$.
If the image $\Phi_{k}(Z)$ is nonvanishing on 
$F^{k+2}+\overline{F^{k+2}}\subset H^{2k+1}(V; \C)$
for $Z\in\CH_k(V)_{\textrm{hom}}$,
then $Z$ is not an element of $\CH_k(V)_{\textrm{alg}}$.

\end{lem}
\begin{proof}
We have only to prove that
the image of $\Phi_{k}$ on $\CH_k(V)_{\textrm{alg}}$
vanishes on $F^{k+2}+\overline{F^{k+2}}\subset H^{2k+1}(V; \C)$.
If an algebraic cycle $Z$ is algebraically equivalent to zero in $V$,
then there exists a topological $(2k+1)$-chain $W$ such that $\partial W=Z$ and $W$ lies on $S$,
where $S$ is an algebraic (or complex analytic) subset of $V$ of complex dimension $k+1$.
The chain $W$ is unique up to $(2k+1)$-cycles.
We may assume that $\omega$ consists of the elements
of $H^{p,q}=F^{p}\cap \overline{F^{q}}$ for $p>k+1$ or $q>k+1$.
Then we have
$\int_{W}\omega =0$.
\end{proof}

From Lemma \ref{isomorphism of real tori}, we have
$J_{k}(V)\cong \Hom_{\Z}(H^{2k+1}(V; \Z),\R/{\Z})$.
Suppose that $V$ is the Jacobian variety $J=J(X)$.
Using the identification $H^{2k+1}(J; \Z)=\wedge^{2k+1}H^{1}$,
we may consider the Abel-Jacobi map of Griffiths as the homomorphism
\[
\Phi_k\colon \CH_k(J)_{\textrm{hom}}
\to \Hom(\wedge^{2k+1}H^{1},\R/{\Z}).
\]
We consider $k=1$ in this subsection.
Let $\nu^{\ast}$ denote the Abel-Jacobi image\index{Abel-Jacobi image} $\Phi_{1}(W_1-W_1^{-})$.
Harris (\cite{0527.30032}, \cite[Proposition 2.1]{1063.14010}) proved that
$(\wedge^3 H^1)^{\prime}$ can be identified with the primitive subgroup of $H^3(J;\Z)=\wedge^3 H^1$ in the sense of Lefschetz,
denoted by $H^3_{\mathrm{prim}}(J;\Z)$.
By this identification and the natural projection
$\Hom_{\Z}(H^{3}(J;\Z),\R/{\Z})\to
\Hom_{\Z}(H^3_{\mathrm{prim}}(J;\Z),\R/{\Z})$,
we consider $\nu^{\ast}$
as an element of
$\Hom_{\Z}((\wedge^3 H^1)^{\prime},\R/{\Z})$.
From Proposition \ref{another definition of harmonic volume},
the 3-cycle $c_3$ in the torus $\R^3/{\Z^3}$ can be identified with $X$ in $J$.
\begin{thm}[\cite{0527.30032,1063.14010}]
\label{Abel-Jacobi}
The Abel-Jacobi image
$\nu^{\ast}$ equals the harmonic volume $\nu_{X}=2I_{X}$.
\end{thm}
This theorem and Lemma \ref{vanishing condition} give us
\begin{prop}\label{cycle and intermediate Jacobian}
If there exists an $\omega\in \wedge^3 H^1
\cap (\wedge^3 H^{1,0}+\wedge^3 H^{0,1})$ such that
$\nu_{X}(\omega)$ is nonzero  modulo $\Z$, then
$W_1-W_1^{-}\not\in \CH_1(J)_{\textrm{alg}}$, {\it i.e.},
$W_1-W_1^{-}$ is algebraically nontrivial in $J$.
\end{prop}
We remark that $F^{3}H^3(J;\C)=\wedge^3 H^{1,0}$.
Similarly, we obtain a sufficient condition that $W_k-W_k^-$ is algebraically nontrivial in $J$.
See Faucette \cite{0783.14003} and Otsubo \cite{1236.14009}.

\section{Pointed Torelli Theorem}
\label{Pointed Torelli Theorem}
We explain the relationship between the pointed harmonic volume and the congruence group
of an extension of mixed Hodge structures.
The classical Torelli theorem states that a compact Riemann surface $X$ is determined by its Jacobian variety $J(X)$ regarded as a complex torus.
We introduce an extension of this theorem for pointed compact Riemann surfaces.
In order to state the theorem, we need to define a mixed Hodge structure on $\Z$-module of finite rank using Chen's iterated integrals.
The pointed harmonic volume plays an important role for the proof of the above theorem.
It determines the moduli space of pointed compact Riemann surfaces.
In the latter half of this section, we give the higher Abel-Jacobi image
$\Phi_{k}(W_{k}(x)-W_{k}^{\pm}(y))$.
This suggests an algorithm in proving the nontriviality condition for the higher Ceresa cycles.

A {\it mixed Hodge structure}\index{mixed Hodge structure} on a $\Z$-module $H_{\Z}$ of finite rank consists of two filtrations, an increasing weight filtration $W_{p}H_{\Q}$ and a decreasing Hodge filtration $F^{p}H_{\C}$ which induces a $\Q$-Hodge structure of weight-$k$ on each graded piece
\[
\mathrm{Gr}_{k}^{W}H_{\Q}=W_{k}H_{\Q}/{W_{k-1}H_{\Q}}.
\]
Here $H_{\Q}$ denotes $H_{\Z}\otimes_{\Z}\Q$.
For mixed Hodge structures $A$ and $B$,
let $\Ext(A,B)$ denote the group of congruence
classes of extensions of mixed Hodge structures,
{\it i.e.}, exact sequences
\[
\SelectTips{cm}{}\xymatrix{
0 \ar[r] & B
\ar[r]^{\beta} & H
\ar[r]^{\alpha} & A
\ar[r] & 0}
\]
of mixed Hodge structures with the natural equivalence relation and Baer sum.
If $v-w=-1$ for the weight $v$ and $w$ of Hodge structures $A$ and $B$ respectively,
then we have the isomorphism
\begin{equation}
\label{extensions of MHS}
\tau\colon \Ext(A,B)\cong J\Hom(A,B).
\end{equation}
Let $E$ denote a congruence class of an extension of $A$ by $B$.
Choose a retraction $r_{\Z}\colon H_{\Z}\to B_{\Z}$, {\it i.e.},
$r_{\Z}\circ \beta=\mathop{\rm id}_{B_{\Z}}$ defined over $\Z$ and
a section $s_{F}\colon A\to H$, {\it i.e.},
$\alpha \circ s_{F} =\mathop{\rm id}_{A}$
which preserves the Hodge filtration.
Then $\tau(E)\in J\Hom(A,B)$ is represented by
$r_{\C}\circ s_{F}\in \Hom(A,B)_{\C}$.
See Carlson \cite{0471.14003} for the general case.

We concentrate on the compact Riemann surface $X$.
Using the augmentation ideal $J_{x}$ of $\Z\pi_1(X,x)$,
we obtain the short exact sequence
\[
\SelectTips{cm}{}\xymatrix{
0 \ar[r] & J_{x}^{2}/{J_{x}^{3}}
\ar[r] & J_{x}/{J_{x}^{3}}
\ar[r] & J_{x}/{J_{x}^{2}}
\ar[r] & 0.
}
\]
The natural map
\[\Z\pi_1(X, x)\ni \gamma \mapsto \gamma -1 \in J_{x}\]
induces an isomorphism $H_1\cong J_{x}/{J_{x}^{2}}$ for any point $x\in X$.
For the additive group $M$, $M^{\ast}$ denotes $\Hom_{\Z}(M,\Z)$.
It is known that $\left(J_{x}^{2}/{J_{x}^{3}}\right)^{\ast}$
is isomorphic to $K\subset H^1\otimes H^1$ which is the kernel of
$(\ , \ ) :H^1\otimes H^1 \to \Z$ induced by the intersection pairing.
We obtain the dual short exact sequence
\begin{equation}
\label{dual short exact sequence}
\SelectTips{cm}{}\xymatrix{
0 \ar[r] & H^1
\ar[r]^<(.25){\alpha} & \left(J_{x}/{J_{x}^{3}}\right)^{\ast}
\ar[r]^>(.75){\beta} & K
\ar[r] & 0.}
\end{equation}
From Corollary \ref{de Rham theorem},
$\left(J_{x}/{J_{x}^{3}}\right)^{\ast}$
is regarded as a sublattice of $H^0(\overline{B}_{2}(X), x)$.
In this situation,
the image $\alpha(\omega)$ is a linear functional $\int \omega$,
where $\omega\in H^{1}$ is a real harmonic 1-form with integral periods.
We give the homomorphism $\beta$ as follows.
Let $\varphi\in \left(J_{x}/{J_{x}^{3}}\right)^{\ast}$ be an element.
Take a homotopy functional $\overline{\varphi}$ on $\Z\pi_{1}(X,x)$
as in (\ref{homotopy functional}).
Here $a_i, b_i\in H^1$.
We set $\beta(\varphi)=\sum a_i\otimes b_i\in K$.

In the general case, Hain \cite{0654.14006} defined a mixed Hodge structure on the truncation of the fundamental group ring of complex manifolds.
We explain the mixed Hodge structure on $\left(J_{x}/{J_{x}^{3}}\right)^{\ast}$
as a sublattice of $H^0(\overline{B}_{2}(X), x)$.
For $i=0,1,2$, the weight filtration $W_i$ is obtained by $H^0(\overline{B}_{i}(X), x)$.
Here $W_i$ is defined over $\Q$.
The Hodge filtration $F^{p}$ is induced by $n$ holomorphic 1-forms on $X$ with $n\geq p$
in the linear functional of $H^0(\overline{B}_{2}(X), x)$ defined over $\C$.
See Pulte \cite{0678.14005} for details.
Using Poincar\'e duality,
Pulte \cite[Lemma 3.7]{0678.14005} obtained the isomorphism of Hodge structures
\[
\Hom(K, H^1)\to \Hom(K\otimes H^1, \C).
\]
By combining (\ref{extensions of MHS}) and this isomorphism,
Lemma \ref{isomorphism of real tori} gives us
\begin{thm}[{\cite[Theorem 3.9]{0678.14005}}]
We have an isomorphism
\[
\Ext(K, H^1)\ni m_{x} \mapsto
I_{(X,x)}\in \Hom_{\Z}(K\otimes H^1, \R/{\Z}).
\]
\end{thm}
The Abel-Jacobi image $\nu^{\ast}=\Phi_1(W_1(x)-W_{1}^{\pm}(y))$ is an element of the intermediate Jacobian $J_1(J)=\Hom_{\Z}(\wedge^3 H^1, \R/{\Z})$.
A natural homomorphism $K\otimes H^1\to \wedge^3 H^1$ induces an injective homomorphism
\[
\Hom_{\Z}(\wedge^3 H^1, \R/{\Z})
\to \Hom_{\Z}(K\otimes H^1, \R/{\Z}).
\]
Pulte extended Theorem \ref{Abel-Jacobi}.
\begin{thm}[{\cite[Theorem 4.9]{0678.14005}}]
The above injection is given by
\[
\Phi_1(W_1(x)-W_{1}^{\pm}(y))
\mapsto
I_{(X,x)}\mp I_{(X,y)}.
\]
\end{thm}

As an application of the pointed harmonic volume for $(X,x)$,
Pulte proved the following pointed Torelli theorem\index{pointed Torelli theorem}.
\begin{thm}[{\cite[Theorem 5.5]{0678.14005}}]
Suppose that $X$ and $Y$ are compact Riemann surfaces and
that $x\in X$ and $y\in Y$.
With the possible exception of two points $x$ in $X$,
if there is a ring isomorphism
\[\Z\pi_1(X,x)/{J_x^3}\to \Z\pi_1(Y,y)/{J_y^3}\]
which preserves the mixed Hodge structure,
then there is a biholomorphism
$\varphi\colon X\to Y$ such that $\varphi(x)=y$.
If $X$ is hyperelliptic or if the harmonic volume is nonzero,
then there are no exceptional points.
\end{thm}
Pointed harmonic volumes determine the structure of pointed compact Riemann surfaces.
In the proof of this theorem,
the classical Torelli theorem follows from the equivalence
of the exact sequence (\ref{dual short exact sequence}),
which preserves the mixed Hodge structure,
and therefore we obtain the biholomorphism $X \cong Y$.
Pulte noted that the pointed Torelli theorem can fail for very special curves
where the harmonic volume $\nu_X$ is zero but the curve is nonhyperelliptic.
The zero locus of the harmonic volume for nonhyperelliptic curves is unknown.

Let $k$ be an integer such that $1\leq k\leq g-2$.
Otsubo \cite[Proposition 3.7]{1236.14009} obtained a reduction of
the higher Abel-Jacobi image $\Phi_{k}(W_{k}(x)-W_{k}^{\pm}(y))
\in J_{k}(J)$ to the case $k=1$,
using the commutative diagram (\ref{AJ commutative diagram}).
We recall $J_{k}(J)\cong \Hom(\wedge^{2k+1}H^1, \R/{\Z})$.
\begin{prop}
\label{higher Abel-Jacobi image}
Suppose that $\omega_i\in H^1$ for $1\leq i\leq 2k+1$.
We have
\begin{align*}
k!\,\Phi_{k}&(W_{k}(x)-W_{k}^{\pm}(y))(\omega_1\wedge \omega_2\wedge \cdots
\wedge \omega_{2k+1})\\
&=k! \sum_{\sigma}\Phi_1(W_{1}(x)-W_{1}^{\pm}(y))
(\omega_{\sigma(1)}\wedge \omega_{\sigma(2)}\wedge \omega_{\sigma(3)})
\prod_{i=1}^{k-1}(\omega_{\sigma(2i+2)},\omega_{\sigma(2i+3)}),
\end{align*}
where $\sigma$ runs through the elements of the (2k+1)-th symmetric group
$S_{2k+1}$ such that
\[
\sigma(1)<\sigma(2)<\sigma(3),
\sigma(2i+2)<\sigma(2i+3)\text{ for }1\leq i \leq k-1,
\]
and
\[
\sigma(2i+2)<\sigma(2i+4)\text{ for }1\leq i \leq k-2.
\]
\end{prop}
Faucette \cite{0764.14015} defined the higher-dimensional harmonic volume\index{harmonic volume!higher-dimensional} as follows.
Suppose $g\geq 2k+1$ for a fixed natural number $k$.
Let $\omega_1,\omega_2,\ldots,\omega_{2k+1}\in H^1$ be real harmonic 1-forms
on $X$ with $\Z$-periods.
A homomorphism $A_{k}
\colon X_{k}\to \R^{2k+1}/{\Z^{2k+1}}$ is defined by
\[
A_{k}(x_1,x_2,\ldots,x_k)
=\left(
\sum_{i=1}^{k}\int_{x}^{x_i}\omega_1,
\ldots,\sum_{i=1}^{k}\int_{x}^{x_i}\omega_{2k+1}
\right).
\]
We define $\omega=\omega_1\wedge \omega_2\wedge \ldots \wedge \omega_{2k+1}\in
\wedge^{2k+1}H^1$ to be a {\it good form} on $J$ if
there exists a topological $(2k +1)$-chain $C_{2k +1}$
in $\R^{2k+1}/{\Z^{2k+1}}$
whose boundary is the image $A_{k}(X_{k})$.
The higher-dimensional harmonic volume $I^{k}_{X}$ is defined to be
the linear functional on the good forms with values in $\R/{\Z}$
\[
I^{k}_{X}(\omega):=\int_{C_{2k+1}}dx_1\wedge dx_2\wedge \cdots
\wedge dx_{2k+1}\quad \mathrm{mod} \ \Z,
\]
where the $\omega_i$ is the pullback of $dx_i$ by $A_{k}$.
The definition is independent of the choice of base point $x$.
Faucette proved that $\Phi_{k}(W_{k}-W_{k}^{-})$
equals $2I^{k}_{X}$ as a linear functional on good forms.
In \cite{0764.14015}, he evaluated $I^2_{X}$ for an unramified double covering of the Fermat quartic $F_4$, and proved that the second Ceresa cycle $W_2-W_2^{-}$ for the curve is algebraically nontrivial.

\section{Nontrivial algebraic cycles in the Jacobian varieties}
\label{Nontrivial algebraic cycles in the Jacobian varieties}
In Section \ref{Abel-Jacobi maps and harmonic volume},
we give a sufficient condition for the $k$-th Ceresa cycle
$W_k-W_k^{-}\in \CH_k(J)_{\textrm{hom}}$ to be algebraically nontrivial.
Harris \cite{0523.14006} proved the nontriviality for the Fermat quartic $F_4$.
Bloch \cite{0527.14008} further studied the Fermat quartic $F_4$ in the context of number theory.
Using methods similar to Harris, the author \cite{1222.14058, 1184.14018}
proved the nontriviality for the Klein quartic $K_4$ and Fermat sextic $F_6$.
Otsubo \cite{1236.14009} extended these results using number theory and obtained an algorithm in proving the nontriviality for the $N$-th Fermat curve $F_{N}$ for $N\geq 4$.
Using his method, we recently obtained the nontriviality for some cyclic quotients of $F_N$ \cite{1009.0096}.
We remark that all the above curves have a representation as a branched cyclic covering of $\C P^1$ and certain values of the harmonic volume for this kind of curve are known to be computable.
Nevertheless, all the values of the harmonic volume for a specific curve are unknown.

\subsection{Fermat curves}
\label{Fermat curves}
We give some value of the higher Abel-Jacobi image
$\Phi_{k}(W_{k}-W_{k}^{-})$ for the Fermat curve as an iterated integral form.
In the former half of this subsection, we have compiled some basic facts on the Fermat curve.

For $N\in \Z_{\geq 4}$, let $F_N=\{(X:Y:Z)\in\C P^2; X^N+Y^N=Z^N\}$
denote the Fermat curve\index{Fermat curve} of degree $N$,
a compact Riemann surface of genus $(N-1)(N-2)/{2}$.
Let $x$ and $y$ denote $X/{Z}$ and $Y/{Z}$ respectively.
The equation $X^N+Y^N=Z^N$ induces $x^N+y^N=1$.
Let $\zeta$ denote $\mathrm{exp}(2\pi\sqrt{-1}/{N})$.
Holomorphic automorphisms $\alpha$ and $\beta$ of $F_N$ are defined as
$\alpha(X:Y:Z)=(\zeta X:Y:Z)$ and $\beta(X:Y:Z)=(X:\zeta Y:Z)$ respectively.
Let $\mu_{N}$ be the group of $N$-th roots of unity in $\C$.
We have $\alpha\beta=\beta\alpha$ and the subgroup, denoted by $G$, of holomorphic automorphisms of $F_N$ generated by $\alpha$ and $\beta$ is isomorphic to $\mu_{N}\times \mu_{N}$.
Let $\gamma_0$ be a path $[0,1]\ni t\mapsto (t,\sqrt[N]{1-t^N})\in F(N)$,
where $\sqrt[N]{1-t^N}$ is a real nonnegative analytic function on $[0,1]$.
A loop in $F_N$ is defined by
$$\kappa_0=\gamma_0\cdot(\beta\gamma_0)^{-1}\cdot
(\alpha \beta\gamma_0)\cdot(\alpha\gamma_0)^{-1},$$
where the product $\ell_1\cdot \ell_2$
indicates that we traverse $\ell_1$ first, then $\ell_2$.
We consider a loop $\alpha^i\beta^j\kappa_0$
as an element of the first homology group $H_1(F_N;\Z)$ of $F_N$,
which is well-known to be a cyclic $G$-module {\cite[Appendix]{0418.14023}}.

Let $\mathbf{I}$ be an index set
$\{(a,b)\in (\Z/{N\Z})^{\oplus 2}; a,b, a+b\neq 0\}$.
For $a\in \Z/{N\Z}\setminus \{0\}$,
we denote its representative by $\langle a\rangle\in \{1,2,\ldots,N-1\}$.
A differential 1-form on $F_N$ is defined by
\[\omega_{0}^{a,b}
=x^{\langle a\rangle -1} y^{\langle b\rangle -1}dx/{y^{N-1}}.\]
Set $\mathbf{I}_{\mathrm{holo}}
=\{(a,b)\in \mathbf{I}; \langle a\rangle+\langle b\rangle<N\}$.
It is well-known that
\[\{\omega_{0}^{a,b}; (a,b)\in \mathbf{I}\}\},\
\{\omega_{0}^{a,b}; (a,b)\in \mathbf{I}_\mathrm{holo}\},\text{ and }
\{\omega_{0}^{-a,-b}; (a,b)\in \mathbf{I}_\mathrm{holo}\},\]
are bases of $H^1(F(N); \C), H^{1,0}(F(N))$, and $H^{0,1}(F(N))$,
respectively.
See Lang \cite{0513.14024} for example.
Clearly,
\[\int_{\alpha^i\beta^j\gamma_0}\omega_{0}^{a,b}
=\zeta^{ai+bj}\int_{\gamma_0}\omega_{0}^{a,b}
=\zeta^{ai+bj}\frac{B(\langle a\rangle/N,\langle b\rangle/N)}{N}.\]
The beta function $B(u,v)$ is defined by $\displaystyle\int_0^1t^{u-1}
(1-t)^{v-1}dt$ for $u,v>0$.
We set $B^N_{a,b}=B(\langle a\rangle/N,\langle b\rangle/N)$.
The integral of $\omega_{0}^{a,b}$
along $\alpha^i\beta^j\kappa_0$ is obtained as follows,
\begin{prop}[{\cite[Appendix]{0418.14023}}]
\label{periods}
We have
\[\int_{\alpha^i\beta^j\kappa_0}\omega_{0}^{a,b}
=B^N_{a,b}(1-\zeta^a)(1-\zeta^b)\zeta^{ai+bj}/{N}.\]
\end{prop}
We denote the $1$-form
$N\omega_{0}^{a,b}/{B^N_{a,b}}$
by
$\omega^{a,b}$.
This implies
$\displaystyle \int_{\alpha^i\beta^j\kappa_0}\omega^{a,b}
\in \Z[\zeta]$.

Let $\Q(\mu_{N})$ be the $N$-th cyclotomic field,
and $\mathcal{O}$ its integer ring.
For a $\Z$-module $M$, we denote the $\mathcal{O}$-module
$M_{\mathcal{O}}$ by $M \otimes_{\Z} \mathcal{O}$.
For each embedding $\sigma\colon \Q(\mu_{N})\hookrightarrow \C$,
let $h\in (\Z/{N\Z})^{\ast}$ be the element such that
$\sigma(\xi)=\zeta^{h}$,
where $\xi$ is a primitive $N$-th root of unity.
Proposition \ref{periods} gives us the element $\varphi^{a,b}\in H_{\mathcal{O}}$
whose image by $H_{\mathcal{O}}\subset H_{\Q(\mu_{N})}\hookrightarrow H_{\C}$
is $\omega^{ha,hb}$.
In a straightforward manner, we obtain the intersection pairing.
\begin{prop}[{\cite[Proposition 4.2]{1236.14009}}]
\label{intersection pairing}
\[
(\varphi^{a,b},\varphi^{c,d})
=\left\{
\begin{array}{ll}
N^2\dfrac{(1-\xi^a)(1-\xi^b)}{1-\xi^{a+b}} & \text{if }
(a,b)=(-c,-d),\\
0 & \text{otherwise}.
\end{array}
\right.
\]
\end{prop}
It immediately follows that
$\varphi^{a,b}\otimes \varphi^{c,d}\in K_{\mathcal{O}}\subset
H_{\mathcal{O}}^{\otimes 2}$ if only and if $(a,b)\neq (-c,-d)$.

The harmonic volume naturally extends to
\[\nu_{\mathcal{O}}=
2I_{\mathcal{O}}\colon (\wedge^3 H)^{\prime}_{\mathcal{O}}\to
(\mathcal{O}\otimes  \R)/{\mathcal{O}}.\]
We have a natural isomorphism
\[\mathcal{O}\otimes\R 
\cong \left[\prod_{\sigma\colon \Q(\mu_{N})\hookrightarrow \C}\C\right]^{+}
\]
where $\sigma$ runs through the embedding of $K$ into $\C$ and $+$ denotes the fixed part under complex conjugation acting on both $\{\sigma\}$ and $\C$ at the same time.
Let $\nu_{\sigma}$ denote the $\sigma$-component of the harmonic volume $\nu_{\mathcal{O}}$.
Let $\mathop{\rm Tr}\colon (\mathcal{O}\otimes \R)/{\mathcal{O}}\to \R/{\Z}$
be the trace map.
We obtain
$\mathop{\rm Tr}\circ \nu_{\mathcal{O}}
=\sum_{\sigma\colon \Q(\mu_{N})\hookrightarrow \C} \nu_{\sigma}$.
To prove the nonvanishing of $\nu_{\mathcal{O}}$,
it is sufficient to prove that of $\mathop{\rm Tr}\circ \nu_{\mathcal{O}}$.
We restrict ourselves to the following situation.
\begin{ass}
\label{assumption}
The elements $(a_i,b_i)\in \mathbf{I}$ $(i=1,2,3)$ satisfy the following conditions:
\begin{enumerate}
\item $\sum_{i=1}^{3}(a_i,b_i)=(0,0)$.
\item For any $h\in (\Z/{N\Z})^{\ast}$, we have either
$(ha_i,hb_i)\in \mathbf{I}_\mathrm{holo}$ for $i=1,2$,
or $(ha_i,hb_i)\not\in \mathbf{I}_\mathrm{holo}$ for $i=1,2$.
\end{enumerate}
\end{ass}
From the second assumption, the correction term $\eta$ in the harmonic volume becomes zero.
Then we have only to compute the iterated integral part.
Using the first assumption and \cite[Theorem 3.6]{1184.14018},
Otsubo \cite{1236.14009} proved
\begin{thm}
\label{trace image}
Under Assumption \ref{assumption}, we obtain
\[
\mathop{\rm Tr}\circ \nu_{\mathcal{O}}
\left(\dfrac{\varphi^{a_1,b_1}\otimes \varphi^{a_2,b_2}\otimes
\varphi^{a_3,b_3}}{(1-\xi^{-a_3})(1-\xi^{-b_3})}\right)
=2N^2\sum\int_{\gamma_0}\omega^{ha_1,hb_1}\omega^{ha_2,hb_2},
\]
where the sum is taken over $h\in (\Z/{N\Z})^{\ast}$
such that $(ha_i,hb_i)\in \mathbf{I}_\mathrm{holo}$ for $i=1,2$.
\end{thm}

We take an element $\varphi=\varphi_1\wedge \varphi_2 \wedge
\cdots \wedge \varphi_{2k+1}\in \wedge^{2k+1}H_{\mathcal{O}}$ for $1\leq k\leq g-2$
in the following way.
Let $(a_i,b_i)\in \mathbf{I}$ $(1\leq i\leq 2k+1)$ be distinct elements such that $(a_i,b_i)$ $(i=1,2,3)$ satisfy Assumption \ref{assumption} and $(a_{2i+2},b_{2i+2})=(-a_{2i+3},-b_{2i+3})$ for $1\leq i\leq k-1$.
We let
\begin{center}
$\varphi_1=\varphi^{a_1,b_1}$,\,
$\varphi_2=\varphi^{a_2,b_2}$,\,
$\varphi_3=\dfrac{1}{(1-\xi^{-a_3})(1-\xi^{-b_3})}\varphi^{a_3,b_3}$,
\end{center}
and
\begin{center}
$\varphi_{2i+2}=
\dfrac{1-\xi^{a_{2i+2}+b_{2i+2}}}{(1-\xi^{-a_{2i+2}})(1-\xi^{-b_{2i+2}})}
\varphi^{a_{2i+2},b_{2i+2}}$,\,
$\varphi_{2i+3}=\varphi^{a_{2i+3},b_{2i+3}}$
\end{center}
for $1\leq i\leq k-1$.
Set $\varphi=\varphi_1\wedge \varphi_2 \wedge\cdots \wedge \varphi_{2k+1}$.
Proposition \ref{intersection pairing} induces $(\varphi_{2i+2},\varphi_{2i+3})=N^2$.
Theorem \ref{trace image} and Proposition \ref{higher Abel-Jacobi image} give us
\begin{cor}
\label{k-th Abel-Jacobi image}
For $1\leq k\leq g-2$, we have
\[
k!\, \mathop{\rm Tr}\circ \Phi_{k}(W_k-W_k^{-})(\varphi)
=
k!\, 2N^{2k}\sum
\int_{\gamma_0}\omega^{ha_1,hb_1}\omega^{ha_2,hb_2},
\]
where the sum is as in Theorem \ref{trace image}.
\end{cor}

 \subsection{A generalized hypergeometric function}
\label{A generalized hypergeometric function}
For the numerical calculation,
we recall the generalized hypergeometric function ${}_3F_2$.
The higher Abel-Jacobi image in the previous subsection
is presented by the special values of ${}_3F_2$.
We introduce the condition to prove nontriviality for the Ceresa cycle
$W_k-W_k^{-}$ for $F_N$.

The gamma function is defined as
$\Gamma(\tau)=\displaystyle\int_{0}^{\infty}e^{-t}t^{\tau-1}dt$ for $\tau>0$
and the Pochhammer symbol as $(\alpha, n)=\G(\alpha +n)/{\G(\alpha)}$
for any nonnegative integer $n$.
For $x\in\{z\in \C\, ; |z|<1\}$ and $\beta_1,\beta_2\not\in\{0,-1,-2,\ldots\}$,
the generalized hypergeometric function\index{generalized hypergeometric function} ${}_3F_2$ is defined by
\[{}_3F_2{\Big(}
\left.
\begin{array}{c}
\alpha_1,\alpha_2,\alpha_3\\
\beta_1,\beta_2
\end{array}
\right.
;x{\Big)}
=\sum_{n=0}^{\infty}{{(\alpha_1,n)(\alpha_2,n)(\alpha_3,n)}
\over{(\beta_1,n)(\beta_2,n)(1,n)}}\, x^n.
\]
Its radius of convergence is $1$.
However, if $\mathop{\rm Re}(\beta_1+\beta_2-\alpha_1-\alpha_2-\alpha_3) >0$,
then ${}_3F_2$ converges when $|x|=1$.
\begin{prop}\label{hypergeometric function}
For $i=1,2$, we put the 1-forms
$\omega_i=t^{\alpha_i-1}(1-t)^{\beta_i-1} dt$
on the unit interval $[0,1]$.
If $0<\alpha_i,\beta_i<1$ for each $i$, then we have the iterated integral
\[
\int_{[0,1]}\omega_1\omega_2
={{B(\alpha_1+\alpha_2,\beta_2)}\over{\alpha_1}}
{}_3F_2\left(
\left.
\begin{array}{c}
\alpha_1,1-\beta_1,\alpha_1+\alpha_2\\
1+\alpha_1,\alpha_1+\alpha_2+\beta_2
\end{array}
\right.
;1\right).
\]
\end{prop}
The proof of this computation is straightforward.
See \cite[Proposition 3.7]{1184.14018} for example.

For simplicity, we introduce
\[
\Gamma
\left(\begin{array}{c}
a_1,a_2,\ldots,a_n\\
b_1,b_2,\ldots,b_m
\end{array}
\right)
=\dfrac{\Gamma(a_1)\Gamma(a_2)\cdots \Gamma(a_n)}
{{\Gamma(b_1)\Gamma(b_2)\cdots \Gamma(b_m)}}.
\]
Using Dixon's formula \cite{0135.28101} and
Proposition \ref{hypergeometric function},
Otsubo obtained another representation of the above iterated integral
$\int_{[0,1]}\omega_1\omega_2$ given as
\[
\Gamma
\left(\begin{array}{c}
\alpha_1,\beta_2,\alpha_1+\alpha_2,\beta_1+\beta_2\\
\alpha_1+\alpha_2+\beta_2,\alpha_1+\beta_1+\beta_2
\end{array}
\right)
{}_3F_2\left(
\left.
\begin{array}{c}
\alpha_1,\beta_2,\alpha_1+\alpha_2+\beta_1+\beta_2\\
\alpha_1+\alpha_2+\beta_2,\alpha_1+\beta_1+\beta_2
\end{array}
\right.
;1\right).
\]
This induces the computation of iterated integrals of 1-forms on $F_N$.
\begin{prop}
Suppose that $(a_i,b_i)\in \mathbf{I}$.
Let $\alpha_i$ and $\beta_i$ denote $\langle a_i\rangle/{N}$ and $\langle b_i\rangle/{N}$,
respectively.
Then the iterated integral $\int_{\gamma_0}\omega^{a_1,b_1}\omega^{a_2,b_2}$
obtains in the form
\[
\Gamma
\left(\begin{array}{c}
\alpha_1+\alpha_2,\beta_1+\beta_2,\alpha_1+\beta_1,\alpha_2+\beta_2\\
\alpha_2,\beta_1,\alpha_1+\alpha_2+\beta_2,\alpha_1+\beta_1+\beta_2
\end{array}
\right)
{}_3F_2\left(
\left.
\begin{array}{c}
\alpha_1,\beta_2,\alpha_1+\alpha_2+\beta_1+\beta_2\\
\alpha_1+\alpha_2+\beta_2,\alpha_1+\beta_1+\beta_2
\end{array}
\right.
;1\right).
\]
\end{prop}

Put
\[
(a_1,b_1)=(1,-2),\
(a_2,b_2)=(-2,1),\ \text{and }
(a_3,b_3)=(1,1).
\]
These $(a_i,b_i)$'s satisfy Assumption \ref{assumption}.
Moreover, we have either $(ha_i,hb_i)\in \mathbf{I}_\mathrm{holo}$ for
$i=1,2,3$, or $(ha_i,hb_i)\not\in \mathbf{I}_\mathrm{holo}$ for $i=1,2,3$.
It immediately follows that
\begin{equation}
\label{condition}
\varphi\in \left(\wedge^{2k+1}H\cap
(F^{k+2}+\overline{F^{k+2}})\right)\otimes_{\Z} \mathcal{O}
\end{equation}
for any $h\in (\Z/{N\Z})^{\ast}$.
From the above proposition, we have
\[
\int_{\gamma_0}\omega^{h,-2h}\omega^{-2h,h}
=\dfrac{\Gamma(1-\frac{h}{N})^4}{\Gamma(1-\frac{2h}{N})^4}
{}_3F_2\left(
\left.
\begin{array}{c}
\frac{h}{N},\frac{h}{N},1-\frac{2h}{N}\\
1,1
\end{array}
\right.
;1\right)
\]
for $h\in (\Z/{N\Z})^{\ast}$.
Using this and Corollary \ref{k-th Abel-Jacobi image},
we have
\[
k!\, \mathop{\rm Tr}\circ \Phi_{k}(W_k-W_k^{-})(\varphi)
=k!\, 2N^{2k}
\underset{\begin{subarray}{c}
0<h<N/2\\
(h,N)=1
\end{subarray}}{\sum}
\dfrac{\Gamma(1-\frac{h}{N})^4}{\Gamma(1-\frac{2h}{N})^4}
{}_3F_2\left(
\left.
\begin{array}{c}
\frac{h}{N},\frac{h}{N},1-\frac{2h}{N}\\
1,1
\end{array}
\right.
;1\right),
\]
for $1\leq k\leq N-2$.
This value is denoted by $f(N,k)$.
Lemma \ref{vanishing condition} and Condition (\ref{condition}) induce
an algebraically nontrivial condition on the $k$-th Ceresa cycle
$k!(W_k-W_k^-)\in \CH_k(J)_{\textrm{hom}}$.
\begin{thm}
\label{nontrivial condition for Fermat}
For the $N$-th Fermat curve $F_N$ and
integer $k$ such that $1\leq k \leq (N-3)/2$,
if the value $f(N,k)$ is noninteger, then
$k!(W_k-W_k^-)$ is not algebraically equivalent to zero in $J(F_{N})$.
\end{thm}
This indeed holds for $N\leqq 1000, k=1$, or $N\leqq 8$ and any $k$,
following numerical calculations using Mathematica.
Otsubo pointed out a sufficient condition for nontorsion.
If $\int_{\gamma_0}\omega^{h,-2h}\omega^{-2h,h}$ is not an element of $\Q(\mu_{N})$ for some integer $h$ with $0<h<N/{2}$ and $(h,N)=1$,
then $W_k-W_k^{-}\in \CH_k(J)_{\textrm{alg}}$ is nontorsion.

\subsection{Cyclic quotients of Fermat curves}
\label{Cyclic quotients of Fermat curves}
As an extension of Otsubo's result, we give some value of the higher Abel-Jacobi image
$\Phi_{k}(W_{k}-W_{k})$ for a cyclic quotient of a Fermat curve.

For any prime number $N$ such that $N\geq 5$, we define a cyclic quotient of a Fermat curve\index{Fermat curve!cyclic quotient of}
$C^{a,b}_N$ as the projective curve whose affine equation is
\[C^{a,b}_N:=\{(u,v)\in \C^2;
v^N=u^a(1-u)^b\},\]
where integers $a$ and $b$ are coprime and satisfy $0<a,b<N$.
The curve $C^{a,b}_N$ is a compact Riemann surface of genus $(N-1)/2$.
We denote by $\pi\colon F_{N}\to C^{a,b}_{N}$ the $N$-fold unramified covering
$\pi(x,y)=(u,v)=(x^N,x^ay^b)$.
The curve $C^{1,2}_7$ is called the Klein quartic $K_4$.
There is a unique 1-form $\tau^{\alpha,\beta}$ on $C^{a,b}_N$
such that $\pi^{\ast}\tau^{\alpha,\beta}=\omega^{\alpha,\beta}$.
Then $\{\tau^{\lambda a,\lambda b}\}
_{\langle \lambda a\rangle+\langle \lambda a\rangle <N}$
is a basis of $H^{1,0}$ of $C_N^{a,b}$
for an integer $\lambda$ and the above fixed $a$ and $b$.
See Lang \cite{0513.14024} for example.
For an embedding $\sigma\colon K\ni \xi\mapsto \zeta^h\in\C$,
we can choose the unique 1-form $\eta^{\alpha,\beta}$ on $C_N^{a,b}$ so that
$\sigma(\eta^{\alpha,\beta})=\tau^{h\alpha,h\beta}$.

For the rest of this paper, we assume that $N$ is prime with $N=1$ modulo $3$.
There exists an integer $1< m < N-1$ such that $m^2+m+1=0$ modulo $N$.
Set
\[
(a_1,b_1)=(1,m),\
(a_2,b_2)=(m,m^2),\ \text{and }
(a_3,b_3)=(m^2,1).
\]
\begin{lem}
These $(a_i,b_i)$'s satisfy Assumption \ref{assumption}.
Furthermore, the conditions $(ha_i, hb_i)\in \mathbf{I}_{\mathrm{holo}}, i=1,2,3$
are equivalent.
\end{lem}

Put
\[\eta_m
=\dfrac{\eta^{1,m}\wedge \eta^{m,m^2}\wedge
\eta^{m^2,1}}{(1-\xi^{-m^2})(1-\xi^{-1})}
\]
We have
\[\eta_m\in
\left(\wedge^{3}H\cap
(F^{3}+\overline{F^{3}})\right)\otimes_{\Z} \mathcal{O}.
\]
Theorem \ref{trace image} gives us
\begin{prop}
The value of the harmonic volume for $C^{1,m}_{N}$ is
\[
 \mathop{\rm Tr}\circ \nu_{\mathcal{O}}
 \left(\eta_m\right)=
 2N^3\sum \int_{\gamma_0}\omega^{h,hm}\omega^{hm,hm^2},
\]
where the sum is taken over $h\in (\Z/{N\Z})^{\ast}$ such that
$(ha_i,hb_i)\in \mathbf{I}_{\mathrm{holo}}$ for $i=1,2,3$.
\end{prop}
\begin{prop}
 \[\int_{\gamma_0}\omega^{h,hm}\omega^{hm,hm^2}
 =\Gamma
\left(\begin{array}{c}
\frac{N-h}{N},\frac{N-\langle hm^2\rangle}{N}\\[4pt]
\frac{\langle hm\rangle}{N}
\end{array}
\right)^2
{}_3F_2\left(
\left.
\begin{array}{c}
\frac{h}{N},\frac{\langle hm\rangle}{N},\frac{\langle hm^2\rangle}{N}\\
1,1
\end{array}
\right.
1\right)
\]
for an integer $h$ such that $0<h<N$ and
$h +\langle hm\rangle +\langle hm^2\rangle =N $.
\end{prop}
Similar to Theorem \ref{nontrivial condition for Fermat},
we obtain the higher Abel-Jacobi image $\Phi_{k}(W_{k}-W_{k})(\varphi)$
and the sufficient condition that $W_k-W_k^-$ is not algebraically equivalent to zero
in $J(C^{1,m}_{N})$.
See \cite{1009.0096} for details.

\noindent
{\bf Acknowledgements.}
First of all, the author is grateful to Athanase Papadopoulos for valuable advice.
He also thanks Subhojoy Gupta, Takuya Sakasai, and Kokoro Tanaka for reading the manuscript and giving useful comments.
This work was partially supported by Fellowship for Research Abroad of Institute from the National Colleges of Technology and Grant-in-Aid for Young Scientists(B) (Grant No.~25800053).
Part of the work was also done while the author stayed at the Danish National Research Foundation Centre of Excellence, QGM (Centre for Quantum Geometry of Moduli Spaces) in Aarhus University.
He is very grateful for the warm hospitality of QGM.
Finally he would like to thank the referee for valuable comments on the preliminary version of this chapter.

\end{document}